    \documentclass{amsart}

        \usepackage{latexsym}
        \usepackage{amssymb}
        \usepackage{amsmath}
        \usepackage{amsfonts}
        \usepackage{amsthm}
        \usepackage[mathcal]{eucal}
        \usepackage{fancyhdr}
        \usepackage[hypertexnames=false]{hyperref}
       \usepackage[all,knot,poly,2cell]{xy}

\usepackage{times}
\usepackage{euler}

\newcommand{\C}{\mathbb{C}}
\newcommand{\Z}{\mathbb{Z}}
\newcommand{\Coll}{\mathscr{C}}
\renewcommand{\bar}{\overline}
\DeclareMathOperator{\supp}{supp}

\newcommand{\Orb}{\mathscr{O}}
\DeclareMathOperator{\Hom}{Hom}

\DeclareMathOperator{\Ann}{Ann}
\DeclareMathOperator{\im}{im}
\DeclareMathOperator{\coker}{coker}

\newcommand{\Aff}{\mathbb{A}}    
\newcommand{\ideal}{\triangleleft}      
\DeclareMathOperator{\spec}{Spec}

\DeclareMathOperator{\Isom}{Isom}

\newcommand{\an}[1]{{#1}_{\mathrm{an}}}

\newcommand{\et}[1]{{#1}_{\tiny\mbox{\'et}}}
\DeclareMathOperator{\obj}{\mathbf{Obj}}

\renewcommand{\Pr}{\mathbb{P}\,}
       \newcommand{\itemref}[1]{\eqref{#1}}

        \theoremstyle{plain}
        \newtheorem{thm}{Theorem}[section]
        \newtheorem{cor}[thm]{Corollary}
        \newtheorem{lem}[thm]{Lemma}
        \newtheorem{prop}[thm]{Proposition}

        \newtheorem{mainthms}{Theorem}

        \theoremstyle{definition}
        \newtheorem{defn}[thm]{Definition}
        \newtheorem{ex}[thm]{Example}

        \theoremstyle{remark}
        \newtheorem{rem}[thm]{Remark}
       
\newcommand{\SEC}[2]{\underline{\mathrm{Sec}}_{{#1}}({#2})} 
\newcommand{\SCHABS}{\mathbf{Sch}}
\newcommand{\COH}[1]{\mathbf{Coh}\,({#1})}
\newcommand{\COHP}[1]{\mathbf{Coh}_p({#1})}
\newcommand{\HS}[1]{\underline{\mathrm{HS}}_{{#1}}}
 
\newcommand{\SCH}[1]{\SCHABS/{#1}}
\newcommand{\BETSITE}[1]{({#1})_{\tiny{\mbox{\'Et}}}}

\newcommand{\HILB}[1]{\underline{\mathrm{Hilb}}_{#1}} 
\newcommand{\fml}[1]{\mathfrak{{#1}}}
\newcommand{\STEIN}[1]{\mathcal{A}({#1})}
\newcommand{\STEINFAC}[1]{\mathcal{B}({#1})}
\newcommand{\AN}[1]{\mathscr{#1}} 
\newcommand{\HSAN}[1]{\HS{{#1}}^{\mathrm{an}}}

\newcommand{\HILBAN}[1]{\HILB{{#1}}^{\mathrm{an}}}
\newcommand{\curv}{U}

    \title{Generalizing the GAGA Principle}
   \author[J. Hall]{Jack Hall}
    \email{jhall@math.stanford.edu}
    \address{Stanford University\\Department of Mathematics\\Building
      380\\Stanford, California 94305\\U.S.A.} 
    \thanks{We would like to express our sincerest thanks to Brian
      Conrad for his insightful comments and suggestions.} 

\newcommand{\ANABS}{\mathbf{An}}
\newcommand{\pscSHV}[1]{\mathbf{PS}_p({#1})}
\newcommand{\pfdSHV}[1]{\mathbf{PS}_{p,\mathrm{DM}}({#1})} 
\begin{document}
\begin{abstract}
  This paper generalizes the fundamental GAGA results of Serre
  \cite{MR0082175} in three ways---to the non-separated setting, to
  stacks, and to families.  As an application of these results, we
  show that analytic compactifications of $\mathcal{M}_{g,n}$ possessing
  \emph{modular} interpretations are algebraizable.
\end{abstract}
\maketitle
\section{Introduction}\label{ch:intro}
A locally separated and locally of finite type algebraic
$\C$-space $X$ may be functorially analytified to an analytic space
$\an{X}$. More generally, one may 
functorially analytify a locally of finite type $\C$-stack with
locally separated diagonal. It was shown by B. Conrad and M. Temkin in
\cite[Thm.\ 2.2.5]{MR2524597} that the analytification of a locally of
finite type algebraic $\C$-space is an analytic space only if it is locally
separated. In particular, for locally of finite type
$\C$-stacks $Z$ and $X$ with locally separated diagonals, one obtains an
analytification functor
\[
\Hom(Z,X) \to \Hom(\an{Z},\an{X}).
\]
The first main result of this paper is 
\begin{mainthms}\label{thm:homs_GAGA}
  For a locally of finite type Deligne-Mumford $\C$-stack $X$ with
  quasi-compact and separated diagonal, and a proper Deligne-Mumford
  $\C$-stack $Z$, the analytification functor:     
  \[
  \Hom(Z,X) \to \Hom(\an{Z},\an{X})
  \]
  is an equivalence.
\end{mainthms}
In \cite{lurie_tannaka}, J. Lurie proves a related result to Theorem
\ref{thm:homs_GAGA}---the stack $X$ is permitted to be
algebraic (as opposed to Deligne-Mumford), but the diagonal is
assumed to be affine. Thus Theorem
\ref{thm:homs_GAGA} is stronger than what appears in [\emph{loc.\
  cit.}]\ when applied to Deligne-Mumford stacks 
(e.g. schemes and algebraic spaces). Our proof of Theorem
\ref{thm:homs_GAGA} will rely on a version of Serre's GAGA principle 
for non-separated algebraic stacks.  
 
For a locally of finite type algebraic $\C$-stack
$X$,  define the category of \textbf{pseudosheaves} on $X$,
$\pscSHV{X}$, to have objects the pairs $(Z\xrightarrow{s}
X,F)$, where $s$ is a representable, quasi-finite, and separated
morphism of algebraic 
stacks, with $Z$ proper, and $F$ a coherent
$\Orb_Z$-module. A morphism of pseudosheaves $(Z\xrightarrow{s} X,F) \to
(Z'\xrightarrow{s'} X, F')$ is a pair
$(Z'\xrightarrow{t} Z , F \xrightarrow{\phi} t_*F')$, where $t$ is a
finite $X$-morphism and $\phi$ is a morphism of coherent
$\Orb_{Z}$-modules. The category of pseudosheaves on an algebraic
stack $X$ is a non-abelian refinement of the category of
coherent sheaves on $X$.

A morphism of analytic stacks $\sigma : \AN{Z} \to \AN{X}$ is
\textbf{locally quasi-finite} if it is representable by analytic
spaces, and for any $x\in |\AN{X}|$ the fiber $\sigma^{-1}(x)$ is a finite
set of points. Thus, for an analytic stack $\AN{X}$, we may similarly
define a category of analytic pseudosheaves $\pscSHV{\AN{X}}$, with
objects given by pairs $(\AN{Z}\xrightarrow{\sigma} \AN{X},\AN{F})$,
where $\sigma$ is a locally quasi-finite and separated morphism of analytic
stacks, $\AN{Z}$ is a proper analytic stack, and
$\AN{F}$ is a coherent $\Orb_{\AN{Z}}$-module. There is an
analytification functor for pseudosheaves:
\[
\Psi_X : \pscSHV{X} \to \pscSHV{\an{X}}.
\]
Theorem \ref{thm:homs_GAGA} will follow from
\begin{mainthms}\label{thm:GAGA_stacks}
  For any locally of finite type Deligne-Mumford $\C$-stack $X$ with 
  quasi-compact and separated diagonal, the 
  analytification functor   
  \[
  \Psi_X : \pscSHV{X} \to \pscSHV{\an{X}}
  \]
  is an equivalence.
\end{mainthms}
Theorem \ref{thm:GAGA_stacks} admits a useful
variant. For a locally of finite type algebraic $\C$-stack 
$X$, define the full subcategory $\pfdSHV{X} \subset \pscSHV{X}$ to
consist of those pseudosheaves $(Z\xrightarrow{s} X,F)$, where the
algebraic stack $Z$ is Deligne-Mumford. For an analytic stack
$\AN{X}$, one may similarly define a full 
subcategory $\pfdSHV{\AN{X}} \subset \pscSHV{\AN{X}}$. There is
an induced analytification functor:
\[
\widetilde{\Psi}_X : \pfdSHV{X} \to \pfdSHV{\an{X}}. 
\]
\begin{mainthms}\label{thm:GAGA_stacks_VAR1}
  For any locally of finite type algebraic $\C$-stack $X$ with
  quasi-compact and separated diagonal, the analytification functor  
  \[
  \widetilde{\Psi}_X : \pfdSHV{X} \to \pfdSHV{\an{X}}
  \]
  is an equivalence.
\end{mainthms}
This follows from Theorem \ref{thm:GAGA_stacks}, since for any locally
of finite type algebraic $\C$-stack $X$ with quasi-compact and
separated diagonal, there is an open substack $X^0 \subset X$ which is
Deligne-Mumford such that the inclusion $X^0 \hookrightarrow X$
induces an equivalence $\pscSHV{X^0} \cong  \pfdSHV{X}$. Also, if an
algebraic stack $X$ has affine stabilizers, then $\pfdSHV{X} \cong
\pscSHV{X}$. Thus Theorem \ref{thm:GAGA_stacks} can be trivially 
strengthened to show that the functor $\psi_X$ is an equivalence for
algebraic stacks with affine stabilizers.   

What follows will be made precise in
\S\ref{sec:applications_GAGA}. Let $g>1$ and $n\geq 0$. Take
${M}_{g,n}$ (resp. $\AN{M}_{g,n}$) to be 
the algebraic $\C$-stack (resp. analytic stack) of smooth,
$n$-pointed, algebraic (resp. analytic) curves of genus $g$. An
algebraic (resp. analytic) modular compactification of ${M}_{g,n}$
(resp. $\AN{M}_{g,n}$) is a proper algebraic (resp. analytic) stack
$N$ (resp. $\AN{N}$), with finite stabilizers, possessing 
a ``modular'' interpretation, together with a diagram of dense open
immersions $M_{g,n} \hookleftarrow V \hookrightarrow N$ 
(resp. $\AN{M}_{g,n} \hookleftarrow \AN{V} \hookrightarrow  
\AN{N} $). Theorem \ref{thm:GAGA_stacks_VAR1} implies
\begin{mainthms}\label{thm:mod_applications}
  Analytic modular compactifications of
  $\AN{M}_{g,n}$ are algebraizable to algebraic modular compactifications
  of $M_{g,n}$.  
\end{mainthms}
Using Moishezon techniques (as in \cite{MR0260747}) for the
coarse moduli spaces, previously existing technology demonstrates that
it is possible to algebraize the analytic coarse moduli spaces which
are modular compactifications of $\AN{M}_{g,n}$. The Moishezon techniques
fail to algebraize the modular interpretation, as well as the moduli
stack, however. Describing the modular compactifications of $M_{g,n}$ is
an active area of  research in moduli theory, referred to as the
Hassett-Keel program (cf. \cite{2008arXiv0806.3444H},
\cite{MR2500894}, and \cite{2010arXiv1010.3751A}). 

We will now discuss the GAGA principle for \emph{families}. The
Hilbert functor, which parameterizes closed immersions into a fixed
space, always fails to be an algebraic space in the non-separated
setting. This assertion is due to C. Lundkvist and R. Skjelnes 
\cite{MR2369042}. In \cite{JACKRYDHHILBSTKSCH-2010}, a generalization
of the Hilbert functor, the \textbf{Hilbert stack}, was shown to be an
algebraic stack for \emph{non-separated} algebraic stacks. For an 
algebraic stack $Y$, and a  
locally finitely presented morphism of algebraic stacks with
quasi-compact and separated diagonal $X \to Y$, the Hilbert stack $\HS{X/Y}$
assigns to each $Y$-scheme $T$ the groupoid of quasi-finite and
representable morphisms of algebraic stacks $Z \to X\times_Y T$ such
that the composition $Z \to X\times_Y T \to T$ is proper, finitely
presented, flat, and with finite diagonal.

For a morphism of 
analytic stacks $\AN{X} \to \AN{Y}$, the 
\textbf{analytic Hilbert stack} $\HSAN{\AN{X}/\AN{Y}}$ assigns to each
analytic $\AN{Y}$-space $\AN{T}$, the groupoid of
locally quasi-finite morphisms of analytic stacks $\AN{Z} \to \AN{X}
\times_{\AN{Y}} \AN{T}$, such that the composition $\AN{Z} \to \AN{X}
\times_{\AN{Y}} \AN{T} \to \AN{T}$ is proper and flat with finite diagonal.
Hence, given a morphism $X \to Y$ of locally of finite 
type algebraic $\C$-stacks with quasi-compact and separated diagonal,
there is thus an induced morphism of $\an{Y}$-stacks:    
\[
\Phi_{X/Y} : \an{(\HS{X/Y})}\to  \HSAN{\an{X}/\an{Y}}. 
\]
It is important to note that there are no results in the literature
for analytic spaces which exert that $\HSAN{\an{X}/\an{Y}}$ is an
analytic stack in the case that the morphism $\an{X} \to \an{Y}$ is
non-separated. We also prove the following GAGA result for
\emph{families}. 
\begin{mainthms}\label{thm:main_GAGA}
 Let $X \to Y$ be a quasi-separated morphism of locally of finite type algebraic
 $\C$-spaces. Then the natural map  
  \[
  \Phi_{X/Y} : \an{(\HS{X/Y})}\to  \HSAN{\an{X}/\an{Y}} 
  \]
  is an equivalence of $\an{Y}$-stacks. In particular,
  $\HSAN{\an{X}/\an{Y}}$ is an analytic $\an{Y}$-stack.   
\end{mainthms}
We wish to emphasize that Theorems
\ref{thm:homs_GAGA}--\ref{thm:GAGA_stacks_VAR1} and
\ref{thm:main_GAGA} are completely new, even in the case where $X$ is 
a \emph{scheme}.   
\begin{rem}
  The arguments in this paper for Theorem \ref{thm:main_GAGA}
  generalize readily to the setting of a morphism of algebraic stacks $X
  \to Y$. What obstructs a full proof of Theorem \ref{thm:main_GAGA}
  in this generality is a lack of foundations for analytic stacks. More
  precisely, we would require an analog of Corollary
  \ref{cor:rep_diag_anal} for analytic stacks, for which no precise
  reference exists in the current literature. 
\end{rem}
\tableofcontents
\subsection{Notation}
 We introduce some notation here that will be used throughout the
paper. For a category $\Coll$ and $X \in \obj\Coll$, we have the
\textbf{slice} category $\Coll/X$, with objects the morphisms $V \to
X$ in $\Coll$, and morphisms commuting diagrams over $X$, which are
called $X$-morphisms. If the category $\Coll$ has finite limits and $f : Y
\to X$ is a morphism in $\Coll$, then for $(V\to X) \in
\obj(\Coll/X)$, define $V_Y := V \times_X Y$. Given a morphism $p : V' 
\to V$, there is an induced morphism $p_Y : V'_Y \to V_Y$. There is
usually an induced functor $f^* : \Coll/X \to \Coll/Y : (V \to X) \mapsto (V_Y
\to Y)$. Given a $(2,1)$-category $\Coll'$, these notions readily
generalize. 

Given a ringed space $U := (|U|,\Orb_U)$, a sheaf of ideals
$\mathscr{I} \ideal \Orb_U$, and a morphism of ringed spaces $g : V
\to U$, we define the \textbf{pulled back ideal} $\mathscr{I}_V = 
\im(g^*\mathscr{I} \to \Orb_V)\ideal \Orb_V$. 

Fix a scheme $S$, then an algebraic $S$-space is a sheaf $F$ on the big
\'etale site of $S$, $\BETSITE{\SCH{S}}$, such that the diagonal morphism
$\Delta_F : F \to F \times_S F$ is represented by schemes, and there
is a smooth surjection $U \to F$ from an $S$-scheme $U$. An
algebraic $S$-stack is a stack $H$ on $\BETSITE{\SCH{S}}$, such that
the diagonal morphism $\Delta_H : H \to H\times_S H$ is represented
by algebraic $S$-spaces and there is a smooth surjection $U \to H$
from an algebraic $S$-space $U$. We make no separation
assumptions on our algebraic stacks, but all
algebraic stacks that are used in this paper will possess
quasi-compact and separated diagonals, thus all of the results of
\cite{MR1771927} apply. 

Let $\ANABS$ denote the category of analytic spaces. The \'etale
topology on $\ANABS$ is the Grothen\-dieck pretopology generated by local
analytic isomorphisms. An \textbf{analytic stack} is an \'etale stack
over $\ANABS$, with diagonal representable by analytic spaces, admitting a
smooth surjection from an analytic space. 
\section{Examples}
We feel that it is instructive to provide examples illustrating the
necessity of some of the hypotheses of the results contained in this 
paper. 
\begin{ex}
  Let $X$ be the Deligne-Mumford stack $B\Z$, and fix an elliptic
  $\C$-curve $E$. Observe that $H^1(\et{E},\Z) = 0$ and
  $H^1(E(\C),\Z) \neq 0$. Hence, there is a map of analytic stacks
  $\an{E} \to B\an{\Z}$ which is not algebraizable to a map of
  algebraic $\C$-stacks $E \to B\Z$. In particular, we conclude that
  the quasi-compact diagonal assumption which appears in Theorem
  \ref{thm:homs_GAGA} is necessary.
\end{ex}
\begin{ex}
  Fix an elliptic $\C$-curve $E$. Then $H^1(E(\C),E(\C)) =
  E(\C)^{\oplus 2}$ and $H^1(\et{E},E) = E(\C)^{\oplus
    2}_{\mathrm{tors}}$. Let $[\gamma] \in
  H^1(E(\C),E(\C))\setminus H^1(\et{E},E) \neq \emptyset$. Then the
  cohomology class $[\gamma]$ corresponds to an analytic map $\an{E}
  \to B\an{E}$, which is not algebraizable to a map of algebraic $\C$-stacks
  $E \to BE$. Thus, we see that a strengthening of Theorem
  \ref{thm:homs_GAGA} to the setting where $X$ is separated, but not
  Deligne-Mumford, is not possible. 
\end{ex}
Note that given a projective $\C$-scheme $X$
and a locally quasi-finite analytic $\an{X}$-space $\AN{Z}$, it is
easy to produce examples where $\AN{Z}$ is non-algebraizable when
$\AN{Z}$ is not compact. Indeed, analytic open subsets $\an{X}$ which
do not arise as Zariski open subsets of $X$ are such examples. 
\begin{ex}
  Consider $X:=\Pr^1_\C$ and one of its coordinate patches $\imath:\Aff^1_\C
  \hookrightarrow X$. Let $\AN{U} \hookrightarrow \Aff^1_\C$ be the
  complement of the closed unit disc. Note that $U$ is an analytic
  space, and the image of $\AN{U}$ in $\an{X}$ is also open. Let $\AN{Z} =
  \an{X} \amalg_{\AN{U}} \an{X}$, taken in the category of ringed
  spaces. Then $\AN{Z}$ is a compact, but non-Hausdorff analytic
  space. Indeed, $\AN{Z}$ is obtained by gluing two analytic spaces
  along an open analytic subspace, and there is a continuous
  surjection $\an{X}\amalg \an{X} \to \AN{Z}$.   There is
  also a local analytic isomorphism of analytic
  spaces $\AN{Z} \to \an{X}$, but clearly $\AN{Z}$ is not
  algebraizable. In particular, we conclude that the separatedness of
  the map $\AN{Z} \to \an{X}$ is necessary in the definition of
  pseudosheaves, if we would like Theorem \ref{thm:GAGA_stacks} to
  hold. 
\end{ex}
\begin{ex}
  In \cite[App. B]{MR0463157}, there is given an example of a smooth
  projective surface $Y$ over $\C$, with an open subscheme
  $\imath : U\hookrightarrow Y$ which posseses a pair of
  non-isomorphic line bundles $\mathscr{L}$ and $\mathscr{M}$ such
  that the analytified line bundles $\an{\mathscr{L}}$ and
  $\an{\mathscr{M}}$ are isomorphic. Note that since 
  $\mathscr{L}$ (resp. $\mathscr{M}$) is $\Orb_U$-coherent, there is
  a coherent $\Orb_X$-module $\bar{\mathscr{L}}$
  (resp. $\bar{\mathscr{M}}$) together with an isomorphism
  $\jmath^*\bar{\mathscr{L}} \cong \mathscr{L}$
  (resp. $\jmath^*\bar{\mathscr{M}} \cong  \mathscr{M}$). Define the
  smooth, universally closed, and finite type $\C$-scheme $X$ by
  gluing two copies of $Y$ along $U$. Let $p$, $q : Y
  \rightrightarrows X$ denote the two different inclusions of $Y$ into
  $X$. On $\an{Y}$ we have two coherent
  $\Orb_{\an{Y}}$-modules $\an{\bar{\mathscr{L}}}$ and
  $\an{\bar{\mathscr{M}}}$. Note that since there is an induced
  analytic isomorphism $\an{\jmath}^*\an{\bar{\mathscr{L}}} \cong 
  \an{\jmath}^*\an{\bar{\mathscr{M}}}$, we may obtain a coherent
  $\Orb_{\an{X}}$-module $\AN{F}$ such that $\an{p}^*\AN{F} \cong
  \an{\bar{\mathscr{L}}}$ and $\an{q}^*\AN{F} \cong
  \an{\bar{\mathscr{M}}}$. If $\AN{F}$ is algebraizable, then there is
  a coherent $\Orb_Y$-module $F$ together with an analytic isomorphism
  $\an{F} \cong \AN{F}$. In particular, we see that there are 
  induced analytic isomorphisms of coherent sheaves on $\an{Y}$:
  \[
  \an{(p^*F)} \cong \an{p}^*\AN{F} \cong \an{\bar{\mathscr{L}}} \qquad
  \mbox{and} \qquad  \an{(q^*F)} \cong \an{q}^*\AN{F} \cong
  \an{\bar{\mathscr{M}}}. 
  \]
  Since $Y$ is a projective $\C$-scheme,  by GAGA \cite[Exp. XII,
  Thm. 4.4]{SGA1}, the induced 
  isomorphism  $\an{(p^*F)} \cong \an{\bar{\mathscr{L}}}$ (resp.
  $\an{(p^*F)}  \cong \an{\bar{\mathscr{L}}}$) is uniquely
  algebraizable to an algebraic isomorphism $p^*F \cong
  \bar{\mathscr{L}}$ (resp. $q^*F \cong \bar{\mathscr{M}}$). However,
  this implies an isomorphism:
  \[
  \mathscr{L} \cong \jmath^*\bar{\mathscr{L}} \cong \jmath^*p^*F \cong
  \jmath^*q^*F \cong \jmath^*\bar{\mathscr{M}} \cong \mathscr{M},
  \]
  which is a contradiction. Hence, $\AN{F}$ is not algebraizable.
  From $\AN{F}$, we can construct the finite and non-algebraizable
  $\Orb_{\an{X}}$-algebra $\Orb_{\an{X}}[\AN{F}]$. Taking the analytic
  spec of this algebra produces a finite homeomorphism $\AN{Z} \to
  \an{X}$ such that $\AN{Z}$ is non-algebraizable. Hence, it is
  necessary for $\AN{Z}$ to be proper, if we would like any locally
  quasi-finite and separated map $\AN{Z} \to \an{X}$ to be algebraizable.
\end{ex}
\section{Preliminaries}\label{sec:sepGAGA}
In this section, we will recall and prove some results for separated
spaces. The generalizations of the classical GAGA results to separated
algebraic stacks will be dealt with in Appendix \ref{app:clGAGA}. For
a morphism of analytic spaces  
$\AN{X} \to\AN{Y}$, the analytic Hilbert functor
$\HILBAN{\AN{X}/\AN{Y}}$ assigns to each analytic $\AN{Y}$-space
$\AN{T}$, the set of isomorphism classes of closed immersions $\AN{Z}
\hookrightarrow \AN{X}_{\AN{T}}$ such that the composition $\AN{Z} \to 
\AN{T}$ is proper and flat. If $\AN{U}$ is another analytic
$\AN{Y}$-space, then one may also define the
functor $\Hom_{\AN{Y}}(\AN{U},\AN{X})$, which assigns to each analytic 
$\AN{Y}$-space $\AN{T}$, the set of $\AN{T}$-morphisms
$\AN{U}_{\AN{T}} \to \AN{X}_{\AN{T}}$. 

In the case that the morphism $\AN{X} \to
\AN{Y}$ is separated, then J. Bingener \cite{MR597743} showed that the
functor $\HILBAN{\AN{X}/\AN{Y}}$ naturally has the structure of a
separated analytic $\AN{Y}$-space. If, in addition, $\AN{U}$ is proper
and flat over $\AN{Y}$, Bingener also proved [\emph{loc.\ cit.}] that
$\Hom_{\AN{Y}}(\AN{U},\AN{X})$ is a separated analytic
$\AN{Y}$-space.
\begin{rem}
  The analyticity of the absolute Hilbert functor (i.e. when $\AN{Y}$
  is a point), is due to \cite{MR0203082}. 
\end{rem}
The following variant of the $\Hom$-space will be useful in this
paper. Again, we fix an analytic $\AN{Y}$-space. Let $\AN{U}
\to\AN{V}$ be a morphism of analytic $\AN{Y}$-spaces. Define  
$\SEC{\AN{Y}}{\AN{U}/\AN{V}}$ to be the functor which assigns to each
$\AN{Y}$-space $\AN{T}$, the set of analytic sections to the morphism
$\AN{U}_{\AN{T}} \to \AN{V}_{\AN{T}}$. Standard arguments involving
fiber products prove the following 
\begin{prop}\label{prop:analsec}
  Fix an analytic space $\AN{Y}$, and let $\AN{U} \to \AN{V}$ be a
  separated morphism of analytic $\AN{Y}$-spaces, where $\AN{V} \to
  \AN{Y}$ is proper and flat. Then $\SEC{\AN{Y}}{\AN{U}/\AN{V}}$ is a
  separated analytic $\AN{Y}$-space. 
\end{prop}
We will record for future reference the following
\begin{cor}\label{cor:rep_diag_anal}
  Let $\AN{X} \to \AN{Y}$ be a morphism of analytic stacks such that
  the diagonal map $\Delta_{\AN{X}/\AN{Y}}$ is a monomorphism. Then the
  diagonal map $\Delta_{\HSAN{\AN{X}/\AN{Y}}}$ is representable by
  separated analytic spaces. 
\end{cor}
\begin{proof}
  Fix an analytic $\AN{Y}$-space $\AN{T}$. Let $\AN{Z}$, $\AN{Z}'$ be
  locally quasi-finite and separated analytic $\AN{X}\times_{\AN{Y}}
  \AN{T}$-stacks which are both $\AN{T}$-proper and flat. The
  monomorphism condition on the diagonal of the map $\AN{X} \to
  \AN{Y}$ ensures that the diagonals of the maps $\AN{Z}\to\AN{T}$ and
  $\AN{Z}' \to \AN{T}$ are monomorphisms. In particular, the
  analytic stacks $\AN{Z}$ and $\AN{Z}'$ are thus analytic
  \emph{spaces}. It remains to show that
  $\Isom_{\HSAN{\AN{X}/\AN{Y}}}(\AN{Z},\AN{Z}')$ is a separated
  analytic $\AN{T}$-space. There is, however, an  
  inclusion of \'etale $\AN{T}$-sheaves
  \[
  \Isom_{\HSAN{\AN{X}/\AN{Y}}}(\AN{Z},\AN{Z}') \hookrightarrow
  \SEC{\AN{T}}{\AN{Z}\times_{\AN{X}} \AN{Z}'/\AN{Z}},
  \]
  which, by \cite[\S10.1, Prop.\ 1]{MR0203082}, is representable by
  open embeddings of analytic spaces. Applying Proposition
  \ref{prop:analsec}, we conclude the result. 
\end{proof}
Using GAGA for separated and locally of finite type $\C$-stacks (cf. 
Theorem \ref{thm:clGAGA} and Corollary \ref{cor:clGAGA1}), the method
of proof employed in 
Corollary \ref{cor:rep_diag_anal} readily proves the following 
\begin{lem}\label{lem:ff_GAGA}
  For a locally of finite type algebraic $\C$-stack $X$ with
  locally separated diagonal, the analytification functor
  \[
  \Psi_X : \pscSHV{X} \to \pscSHV{\an{X}}
  \]
  is fully faithful.
\end{lem}
Given algebraic $\C$-spaces $U$ and $V$, it is a non-trivial
task to determine how the set of analytic morphisms between $\an{U}$
and $\an{V}$ relates to the set of algebraic morphisms between $U$ and
$V$. Since we are interested in analytifications of  
algebraic stacks, which are defined by moduli problems, such a
relation will be essential for us to proceed, however. 
\begin{lem}\label{lem:EXTRAUSEFUL1}
  Fix a locally of finite type algebraic $\C$-stack $V$ with locally
  separated diagonal and a locally of
  finite type algebraic $V$-stack $U$ with locally separated diagonal.
  Let $S$ be a local artinian $V$-scheme, then the analytification functor
    \[
    \Hom_V(S,U) \to \Hom_{\an{V}}(\an{S},\an{U})
    \]
    is an equivalence of categories.
\end{lem}
\begin{proof}
  It suffices to treat the case where $V=\spec \C$. Let $k>0$ be such
  that the maximal ideal $\mathfrak{m}_S$ of the 
  artinian $\C$-algebra $\Gamma(S,\Orb_S)$ satisfies
  $\mathfrak{m}_S^k = 0$. If $U$ is a scheme, then a morphism $S
  \to U$ is equivalent to specifying a point $u\in U(\C)$ and a
  morphism of $\C$-algebras
  $\Orb_{U,u}/\mathfrak{m}_{\Orb_{U,u}}^k \to 
  \Gamma(S,\Orb_S)$. Similarly, an analytic morphism $\an{S}
  \to \an{U}$ is equivalent to specifying a point $u\in \an{U}=U(\C)$,
  as well as a morphism of $\C$-algebras
  $\Orb_{\an{U},u}/\mathfrak{m}_{\Orb_{\an{U},u}}^k \to
  \Gamma(S,\Orb_S)$. Recalling that there is a bijection of
  $\C$-algebras $\Orb_{U,u}/\mathfrak{m}_{\Orb_{U,u}}^k
  \cong \Orb_{\an{U},u}/\mathfrak{m}_{\Orb_{\an{U},u}}^k$ as well as a 
  bijection of sets $\an{U} \to U(\C)$, then we conclude that we're
  done in the case that $U$ is a scheme. If $U$ is an algebraic stack,
  let $U' \to U$ be a smooth cover by a scheme $U'$.  Since $S$ is local 
  artinian, then the lifting criterion for smoothness implies that the
  analytification functor $\Hom(S,U) \to
  \Hom(\an{S},\an{U})$ is essentially surjective. That the
  analytification functor is fully faithful is clear.
\end{proof}
\begin{lem}\label{lem:useful2}
 Fix an analytic stack $\AN{W}$, and consider a morphism of analytic
 $\AN{W}$-stacks $p : \AN{U} \to \AN{V}$. Suppose 
  that for any local artinian $\AN{W}$-scheme $S$, the functor
  \[
  p_*\mid_{\an{S}} : \Hom_{\AN{W}}(\an{S},\AN{U}) \to
  \Hom_{\AN{W}}(\an{S},\AN{V}) 
  \]
  is fully faithful. Then $p$ is a monomorphism of analytic stacks,
  and is, in particular, representable by analytic spaces. If, in
  addition, the functor $p_*\mid_{\an{S}}$ is essentially surjective,
  then $p$ is an analytic isomorphism. 
\end{lem}
\begin{proof}
  First, we assume that the $p$ is representable by analytic spaces
  and $p_*\mid_{\an{S}}$ is an equivalence. The lifting criteria for
  analytic \'etale morphisms shows that $p$ is a bijective \'etale
  morphism, hence an analytic isomorphism. In the case where the map
  $p$ is not assumed to be representable by analytic spaces, and
  $p_*\mid_{\an{S}}$ is fully faithful, we observe that the diagonal
  map $\Delta_p$ is representable by analytic spaces and
  $(\Delta_p)_*\mid_{\an{S}}$ is an equivalence. By what we have shown
  already, the diagonal map $\Delta_p$ is an analytic isomorphism, and
  so the map $p$ is a monomorphism of analytic stacks. Noting that a
  monomorphism of analytic stacks is representable by analytic spaces,
  then we're done.  
\end{proof}
To prove Theorem \ref{thm:main_GAGA}, the following
strengthening of Lemma \ref{lem:useful2} will be useful.
\begin{cor}\label{cor:conrad_crit}
  Fix an analytic stack $\AN{Y}$. Consider a morphism of \'etale
  $\AN{Y}$-stacks  $P : \mathscr{F} \to  
  \mathscr{G}$. Suppose the following conditions are
  satisfied:
  \begin{enumerate}
  \item\label{item:analF} $\mathscr{F}$ is an analytic stack;
  \item\label{item:diagG} the diagonal of $\mathscr{G}$ is
    representable by analytic spaces; 
  \item\label{item:bijART} for any local artinian $\AN{Y}$-scheme $S$,
    the functor 
    \[
    \mathscr{F}(\an{S}) \to \mathscr{G}({\an{S}})
    \]
    is fully faithful (resp. an equivalence).
  \end{enumerate}
  Then $P$ is a monomorphism (resp. an equivalence).
\end{cor}
\begin{proof}
  Conditions \itemref{item:analF} and \itemref{item:diagG} imply that
  the morphism $P$ is representable by analytic stacks. Thus we may
  assume that $\mathscr{G}$ is an analytic space, and $\mathscr{F}$ is
  an analytic stack. Condition \itemref{item:bijART} implies Lemma
  \ref{lem:useful2} applies, and we're done.
\end{proof}
We defer the proof Theorem \ref{thm:GAGA_stacks} to
\S\ref{sec:nsGAGA}, and focus on applying this result to the proof of
Theorem \ref{thm:main_GAGA}.  By Lemma \ref{lem:EXTRAUSEFUL1} and
Corollary \ref{cor:conrad_crit}, to prove Theorem \ref{thm:main_GAGA},
the next result will be of use.   
\begin{lem}\label{lem:artinian_GAGAstacks}
  Fix a locally of finite type algebraic $\C$-space $Y$, and a locally
  of finite type algebraic $Y$-space
  $X$. For any  local artinian $Y$-scheme $S$, the functor
  \[
  \Phi_{X/Y}({\an{S}}) : \HS{X/Y}({S}) \to
 \HSAN{\an{X}/\an{Y}}(\an{S}) 
  \]
  is an equivalence.
\end{lem}
\begin{proof}
  First, we show that $\Phi_{X/S}(\an{S})$ is fully faithful. Given
  quasi-finite and separated maps $Z$, $Z' \rightrightarrows X\times_Y
  S$, where the two compositions $Z$, $Z' \rightrightarrows S$ are
  proper and flat, then regarding $X\times_Y S$
  as a locally of finite type $\C$-space, as well as $Z$ and $Z'$ as
  proper $\C$-spaces, we conclude by the full faithfulness of
  $\Psi_{X\times_Y S}$ that any analytic $\an{X}\times_{\an{Y}}
  \an{S}$-map $\an{Z} \to \an{Z}'$ arises from an algebraic $X\times_Y
  S$-map $Z \to Z'$ over $\C$. Applying GAGA (cf. Corollary
  \ref{cor:clGAGA2}), we conclude that $\C$-maps $Z \to Z'$ 
  which analytify to $\an{S}$-maps $\an{Z} \to \an{Z}'$ are
  automatically $S$-maps---hence we've shown that $\Phi_{X/Y}(\an{S})$
  is fully faithful. 

  To show that the functor $\Phi_{X/Y}(\an{S})$ is essentially
  surjective, we note that given a locally quasi-finite and separated
  morphism of analytic spaces $\AN{Z} \to \an{X}\times_{\an{Y}}
  \an{S}$ such that the composition $\AN{Z} \to \an{S}$ is proper and 
  flat, then again regarding $Z$ and $X\times_Y
  S$ as living over $\C$, then Theorem \ref{thm:GAGA_stacks} implies
  that the morphism $\AN{Z} \to \an{X}\times_{\an{Y}} \an{S}$
  algebraizes to a quasi-finite and separated morphism of algebraic
  $\C$-spaces $Z \to X\times_Y S$ such that $Z$ is
  $\C$-proper. Applying GAGA (cf. Corollary \ref{cor:clGAGA2}) again, we
  may algebraize the flat morphism $\an{Z} \to \an{S}$ to a flat
  morphism $Z \to S$, which proves essential surjectivity of the
  functor $\Phi_{X/S}(\an{S})$. 
\end{proof}
\begin{proof}[Proof of Theorem \ref{thm:main_GAGA}]
  We use the criteria of Corollary \ref{cor:conrad_crit}. For
  \itemref{item:analF}, by \cite[Thm.\
  2]{JACKRYDHHILBSTKSCH-2010}, we have that $\an{(\HS{X/Y})}$ is an
  analytic stack. For \itemref{item:diagG}, by Corollary
  \ref{cor:rep_diag_anal} we deduce that $\HSAN{\an{X}/\an{Y}}$ has
  representable diagonal. For \itemref{item:bijART}, we apply Lemma
  \ref{lem:artinian_GAGAstacks}.
\end{proof}
\section{Non-separated GAGA}\label{sec:nsGAGA}  
To prove Theorem \ref{thm:GAGA_stacks}, we will utilize two methods
of d\'evissage, which are summarized by the next two propositions.
\begin{prop}[Birational d\'evissage]\label{prop:surj_cpoints_dev_case} 
    Suppose that $X$ is a finite type $\C$-scheme. Suppose that 
    for any closed immersion $V \hookrightarrow X$, there is a proper,
    schematic, and birational morphism $V' \to V$ such that the
    functor $\Psi_{V'}$ is an equivalence, then the functor $\Psi_X$
    is an equivalence.
\end{prop}
\begin{prop}[Generically finite \'etale d\'evissage]\label{prop:fet_dev}
  Suppose that $X$ is a finite type Deligne-Mumford $\C$-stack with
  quasi-compact and separated diagonal. Suppose that for any closed
  immersion $V \hookrightarrow X$, there is a finite and generically
  \'etale map $V' \to V$ such that the functors $\Psi_{V'}$ and
  $\Psi_{V'\times_V V'}$ are equivalences, then the functor $\Psi_X$
  is an equivalence.  
\end{prop}
It will also be necessary to have a result on sections to \'etale morphisms.
\begin{prop}\label{prop:etGAGA}
  Consider a quasi-compact and \'etale map of locally
  of finite type algebraic $\C$-spaces $W \to Z$. Then, the map of \emph{sets}
  \[
  \Hom_Z(Z,W) \to \Hom_{\an{Z}}(\an{Z},\an{W})
  \]
  is bijective.
\end{prop}
The proofs of Propositions \ref{prop:surj_cpoints_dev_case} and
\ref{prop:fet_dev} will fill the remainder of this section, and we
will utilize arguments very similar to those in
\cite[\S4]{JACKRYDHHILBSTKSCH-2010}.  With
the above results at our disposal, however, we can prove Theorem
\ref{thm:GAGA_stacks_VAR1} immediately.   
\begin{proof}[Proof of Theorem \ref{thm:GAGA_stacks_VAR1}]
  By Lemma \ref{lem:ff_GAGA}, we know that the functor $\Psi_X$ is
  fully faithful, so it remains to show that the functor
  $\widetilde{\Psi}_X$ is essentially surjective.
 {\par\noindent}\textbf{Basic Case.} First, we suppose that $X$ is a
 finite type $\C$-scheme, and the structure morphism factors as the
 composition $X \xrightarrow{f} Y \xrightarrow{g} \spec \C$, where the
 morphism $f$ is quasi-compact \'etale and the morphism $g$ is
 projective. Fix a locally quasi-finite morphism of analytic spaces   
  $s:\AN{Z} \to \an{X}$ such that $\AN{Z}$ is proper. 
  Since the analytic space $\an{Y}$ is separated, one
  concludes that the composition $f\circ s : \AN{Z} \to \an{Y}$ is a
  locally quasi-finite and proper morphism of analytic spaces. By \cite[Thm.\
  XII.4.2]{MR755331} such a morphism is a finite. Since $Y$ is also 
  projective, by \cite[Exp. XII, 
  Thm. 4.6]{SGA1}, we conclude that there is a finite morphism of
  $\C$-schemes $Z \to Y$ such that $Z$ is proper, and $\an{Z}
  \cong \AN{Z}$ over $\an{Y}$.  To complete the proof in this
  setting, it suffices to produce a quasi-finite morphism of schemes
  $t : Z \to X$ such that $\an{t} = s$.  Note that it is equivalent to 
  produce a section to the quasi-compact \'etale morphism $h :
  X\times_Y Z \to Z$ which agrees with the analytic section to the
  morphism of analytic spaces $\an{h}$ induced by $s$. This is
  precisely the content of Proposition \ref{prop:etGAGA}. Now, if we
  attach to $\AN{Z}$ a coherent $\Orb_{\AN{Z}}$-module $\AN{F}$, then
  since $\AN{Z}$ is algebraizable, GAGA for proper $\C$-schemes
  \cite[Exp. XII, Thm. 4.4]{SGA1} implies that $\AN{F}$ is
  algebraizable. Thus the functor $\Psi_X$ is an equivalence in this case.

{\par\noindent}\textbf{Finite type $\C$-schemes.} Next, we just
assume that $X$ is a finite type $\C$-scheme. Note that any closed
subscheme $V\hookrightarrow X$ is also a finite type
$\C$-scheme. Thus, by \cite[Cor.\ 5.7.13]{MR0308104}, there is a 
schematic and birational morphism $V' \to V$ such that the structure
morphism of $V'$ over $\C$ factors as $V' \xrightarrow{f_V} Y_V
\xrightarrow{g_V} \spec \C$, where $f_V$ is quasi-compact \'etale and
$g_V$ is projective. By the basic case considered above, the functor
$\Psi_{V'}$ is an equivalence and so by Proposition
\ref{prop:surj_cpoints_dev_case}, we conclude that the functor
$\Psi_X$ is an equivalence. 

{\par\noindent}\textbf{Finite type Deligne-Mumford $\C$-stacks.} Here,
we assume that $X$ is a finite type Deligne-Mumford $\C$-stack. Note that if $V
\hookrightarrow X$ is any closed immersion, then $V$ is also a finite
type Deligne-Mumford $\C$-stack. By \cite[Thm.\ 16.6]{MR1771927},
there is a finite and generically \'etale morphism $V' \to V$, where
$V'$ is a scheme. Since $V'$ is thus a finite type $\C$-scheme, we may
conclude by the case previously considered that the functors
$\Psi_{V'}$ and $\Psi_{V'\times_V V'}$ are equivalences. Hence, by
Proposition \ref{prop:fet_dev} we conclude that the functor $\Psi_X$
is an equivalence. 

{\par\noindent}\textbf{General case.} Fix a
  locally quasi-finite and separated morphism of analytic stacks
  $s : \AN{Z} \to \an{X}$ where $\AN{Z}$ is proper with finite
  diagonal. Let $X^0$ denote the open substack of $X$ which has
  quasi-finite diagonal. Note that since $\AN{Z}$ has finite diagonal,
  then the quasi-finite and separated map $s : \AN{Z} \to \an{X}$
  factors canonically through $\an{X}^0$. Replacing $X$ by $X^0$, we
  may henceforth assume that $X$ has quasi-finite and separated
  diagonal. Next, let $O_X$  
  denote the category of quasi-compact open subsets of $X$. We note
  that $\{U\}_{U\in O_X}$ is an open cover of $X$ 
  and so $\{s^{-1}(\an{U})\}_{U\in O_X}$ is an open cover of
  $\AN{Z}$. Since $\AN{Z}$ is a compact topological space, and the
  exhibited cover is closed under finite unions, there is an
  open immersion $U \hookrightarrow X$ such that the map
  $\AN{Z} \to \an{X}$ factors uniquely through $\an{U}$ and $\an{U}$
  is a finite type $\C$-stack with quasi-finite and separated
  diagonal. By the previous case considered, we conclude that the functor
  $\Psi_{X}$ is essentially surjective, thus an equivalence.  
\end{proof}
Given a morphism of ringed spaces $f
: U \to V$, then we say that $f$ is \textbf{Stein} if the map of
sheaves $f^\sharp : \Orb_V \to f_*\Orb_U$ is an isomorphism. By
\cite[\S10.6.1]{MR755331}, if the morphism $f$ is a proper morphism of
analytic spaces, then there is a \textbf{Stein factorization} $U
\xrightarrow{\STEIN{f}} \STEIN{U} \xrightarrow{\STEINFAC{f}} V$, where
the morphism $\STEIN{f}$ is proper, Stein, surjective, with connected
fibers, and the morphism $\STEINFAC{f}$ is finite. Similarly, if the
morphism $f$ is a proper morphism of locally noetherian schemes, then by
\cite[\textbf{III}, 4.3.1]{EGA}, the morphism $f$ has a Stein
factorization $U  \xrightarrow{\STEIN{f}} \STEIN{U}
\xrightarrow{\STEINFAC{f}} V$, whose formation commutes with flat base
change on $V$. What is important here is that by
\cite[Exp. XII, Thm. 4.2]{SGA1}, the formation of the Stein factorization
commutes with analytification. 
\begin{lem}\label{lem:stein_analytic}
  Fix a proper morphism of locally of finite type
  $\C$-schemes $\pi : Y \to X$, and a locally quasi-finite and separated
  morphism of analytic spaces $\sigma : \AN{Z} \to \an{X}$ such that
  $\AN{Z}$ is proper, and the analytic space $\AN{Z}_{\an{Y}}$ is
  algebraizable to a quasi-finite and separated $Y$-scheme
  $W$. In the Stein 
  factorization of the proper morphism $(\an{\pi})_{\AN{Z}} :
  \AN{Z}_{\an{Y}}\to \AN{Z}$ of analytic spaces, $\AN{Z}_{\an{Y}} 
  \xrightarrow{\STEIN{(\an{\pi})_{\AN{Z}}}} \STEIN{\AN{Z}_{\an{Y}}}
  \xrightarrow{\STEINFAC{(\an{\pi})_{\AN{Z}}}} \AN{Z}$, the analytic space
  $\STEIN{\AN{Z}_{\an{Y}}}$ is algebraizable to a quasi-finite and
  separated $X$-scheme.
\end{lem}
\begin{proof}
  By \cite[\textbf{IV}, 18.12.13]{EGA}, there is a finite 
  $Y$-scheme $\bar{t} : \bar{W} \to Y$ and an open immersion $\jmath :
  W \hookrightarrow \bar{W}$. Define the map $\rho := \pi\circ \bar{t} :
  \bar{W} \to X$, which is proper, and consider the Stein
  factorization $\bar{W} \xrightarrow{\STEIN{\rho}}
\STEIN{\bar{W}} \xrightarrow{\STEINFAC{\rho}} X$. We claim that 
$\STEIN{\rho}^{-1}\STEIN{\rho}(|W|) = |W|$. In particular, since
the map $\STEIN{\rho}$ is proper and surjective, it is universally
submersive, thus it will follow that $\STEIN{\rho}(|W|)$ is an open
subset of $|\STEIN{\bar{W}}|$. Since the separated morphism $W \to
\STEIN{\bar{W}}$ has proper fibers, it follows that for any $w\in
\STEIN{\bar{W}}$, the morphism on the geometric fibers over $w$:
$W_{\bar{w}} \to \bar{W}_{\bar{w}}$ is an open and 
closed immersion. Since $\bar{W}_{\bar{w}}$ is geometrically
connected, we conclude that $W_{\bar{w}} = \bar{W}_{\bar{w}}$ or
$\emptyset$, and the 
claim is proved. We let the quasi-finite and separated morphism $s'
:Z' \to X$ denote the induced open subscheme of
$\STEIN{\rho}(\bar{W})$ defined by the image of $|W|$. It now remains
to construct a unique morphism of  
analytic spaces $\alpha : \an{Z}' \to \AN{Z}$ which is compatible with
the induced map $\beta : \an{W} \to \an{Z}'$. Indeed, since $\an{W} \to
\an{Z}'$ is proper and Stein, then $\an{W} \to \an{Z}' \to \AN{Z}$ is
the Stein factorization of $(\an{\pi})_{\AN{Z}}$.

We first define the map $\alpha$ set-theoretically. For $w\in
\an{W}'$, the subset $\beta^{-1}(w) \subset \an{W}$ is closed and
connected, thus since $(\an{\pi})_{\AN{Z}}$ is proper, we deduce that
$u_w:=(\an{\pi})_{\AN{Z}}\beta^{-1}(w)$ is a closed and connected
subset of $|\AN{Z}|$. Moreover, $u_w$ is also
quasi-finite over the image of $w$ in $\an{X}$ and we deduce that
$u_w$ is a single point. Hence, we obtain a well-defined map of sets $\alpha
: |\an{W}'| \to |\AN{Z}|$. Clearly, $\alpha$ is continuous, as 
the maps $(\an{\pi})_{\AN{Z}}$ and $\beta$ are surjective and
submersive. We obtain the map on functions from the following 
composition:
\[
\Orb_{\AN{Z}} \to ((\an{\pi})_{\AN{Z}})_*\Orb_{\an{W}} \cong
\alpha_*\beta_*\Orb_{\an{W}} \cong \alpha_*\Orb_{\an{W}'},
\]
with the last isomorphism because $\beta$ is Stein. 
\end{proof}
We have two easy lemmata.
\begin{lem}\label{lem:AR_analy}
  Fix a compact analytic stack $\AN{X}$, a
  $\Orb_{\AN{X}}$-coherent sheaf $\AN{F}$ and a coherent
  $\Orb_{\AN{X}}$-ideal $\AN{I}$. 
  \begin{enumerate}
  \item \label{lem:AR_analy1} If $|\supp \AN{F}| \subset |V(\AN{I})|$, then
    there is a $k>0$ such that $\AN{I}^k\AN{F} = (0)$.
  \item \label{lem:AR_analy2} Given a coherent subsheaf $\AN{F}'
    \subset \AN{F}$ such that $\AN{I}\AN{F}'=(0)$, then there is a $k
    > 0)$ such that $(\AN{I}^k\AN{F}) \cap \AN{F}' = (0)$.  
  \end{enumerate}
\end{lem}
\begin{proof}
  For \itemref{lem:AR_analy1}, the R\"uckert Nullstellansatz
  \cite[\S3.2]{MR755331} implies that for any $x\in \AN{X}$, there is
  an open neighborhood $\AN{U}_x$ and a $k_x$ such that
  $(\mathscr{I}^{k_x}\AN{F})_{\AN{U}_x} = (0)$. Since $\AN{X}$ is
  compact, we conclude the result. For \itemref{lem:AR_analy2}, fix
  $x\in \AN{X}$ 
  and observe that 
  since the local ring $\Orb_{\AN{X},x}$ is noetherian, by
  \cite[Cor.\ 10.10]{MR0242802} $\exists \,k_x 
  > 0$ such that $(\mathscr{I}^{k_x}_x\AN{F}_{x}) \cap \AN{F}'_x =
  (0)$. Note that for any $x\in \AN{X}$, the sheaf of
  $\Orb_{\AN{X}}$-modules $\AN{G}^x = (\AN{I}^{k_x}\AN{F}) \cap
  \AN{F}'$ is coherent, and has closed support. Thus, as $(\AN{G}^x)_x =
  (0)$, then there is an open neighborhood $U_x$ of $x\in \AN{X}$ such
  that $\AN{G}^x\mid_{U_x} = (0)$. Since $\AN{X}$ is compact in the
  analytic topology, we conclude that there is a finite set of points
  $x_1$, $\dots$, $x_n \in \AN{X}$ such that
  $\AN{G}^{x_i}\mid_{U_{x_i}} = (0)$ and $\{U_{x_i}\}_{i=1}^n$ covers
  $\AN{X}$. Take $k = \max_{i} k_{x_i}$, then $(\AN{I}^k\AN{F})\cap
  \AN{F}' =\bigcap_{i=1}^n \AN{G}^{x_i} = (0)$, as claimed. 
\end{proof}
\begin{lem}\label{lem:analy_noeth_ind}
  Fix a proper morphism of finite type 
  $\C$-schemes $\pi : Y \to X$,  and a locally quasi-finite and
  separated morphism of analytic spaces $\sigma
  : \AN{Z} \to \an{X}$, such that $\AN{Z}$ is compact. Let $U
  \hookrightarrow X$ be an open subscheme such that 
  the induced map $\pi^{-1}U \to U$ is an isomorphism. Fix a coherent
  ideal sheaf ${I} \ideal \Orb_X$ such that 
  $|V(I)| = |X\setminus U|$. The Stein factorization of the proper
  morphism of analytic spaces $(\an{\pi})_{\AN{Z}} : \AN{Z}_{\an{Y}}
  \to \AN{Z}$ induces a map of coherent $\Orb_{\AN{Z}}$-algebras
  $\STEINFAC{(\an{\pi})_{\AN{Z}}}^\sharp : \Orb_{\AN{Z}} \to
  \STEINFAC{(\an{\pi})_{\AN{Z}}}_*\Orb_{\STEIN{(\an{\pi})_{\AN{Z}}}}$. There
  is a $k>0$ such that the coherent $\Orb_{\AN{Z}}$-ideal 
  $(\an{I})^k_{\AN{Z}}$ annihilates the kernel and cokernel of the map 
  $\STEINFAC{(\an{\pi})_{\AN{Z}}}^\sharp$.  
\end{lem}
\begin{proof}
  Define the open subset $\AN{U}' = \sigma^{-1}(\an{U})$, then
  certainly we have that the map
  $\STEINFAC{(\an{\pi})_{\AN{Z}}}^\sharp$ is an
  isomorphism when restricted to the open subset $\AN{U}'$. In
  particular, it follows that the kernel and cokernel of the map
  $\STEINFAC{(\an{\pi})_{\AN{Z}}}^\sharp$ are supported on the analytic
  set $|V((\an{I})_{\AN{Z}})|$. Now apply Lemma
  \ref{lem:AR_analy}\itemref{lem:AR_analy1}.
\end{proof}
For a morphism of ringed topoi $f :
U \to V$, we denote the induced map on structure sheaves by $f^\sharp :
\Orb_V \to f_*\Orb_U$. Define the \textbf{conductor} of the 
morphism $f$ to be $\fml{C}_f := \Ann_{\Orb_V}(\coker f^\sharp)$. The
result that follows is the main d\'evissage technique used in proving
Propositions \ref{prop:surj_cpoints_dev_case} and \ref{prop:fet_dev}.
\begin{lem}\label{lem:inj_surj_analy}
  Consider a finite type $\C$-stack $X$ with locally separated
  diagonal. Suppose that we have a commutative diagram of analytic
  stacks: 
  \[
  \xymatrix@!0{\AN{Z}' \ar[rr]^f \ar[dr]_{s'} & & \AN{Z} \ar[dl]^s\\ &
    \an{X} & }
  \]
  where $\AN{Z}'$ and $\AN{Z}$ are compact. Suppose that $\AN{Z}'$ is
  algebraizable to a quasi-finite and separated algebraic $X$-stack,
  $s$ and $s'$ are locally quasi-finite and separated, and $f$ is finite and
  surjective. 
  \begin{enumerate}
 \item \label{lem:inj_surj_analy2} If there is a coherent ideal
    $\mathscr{I} \ideal \Orb_{\AN{Z}}$ such that $\mathscr{I} \cap
    \ker f^\sharp = (0)$ and $V(\mathscr{I})$ is algebraizable to
    a quasi-finite and separated $X$-stack, then  
    there is a factorization of the morphism $f$ as $\AN{Z}' \to
    \AN{Z}'' \xrightarrow{\beta} \AN{Z}$ such that $\beta$ is finite
    and surjective, $\beta^\sharp : \Orb_{\AN{Z}} \to
    \beta_*\Orb_{\AN{Z}''}$ is injective, $\AN{Z}''$ is algebraizable
    to a quasi-finite and separated algebraic $X$-stack, and
    $\fml{C}_f \subset \fml{C}_\beta$.
  \item \label{lem:inj_surj_analy3} If there is a coherent ideal
    $\mathscr{J} \ideal \Orb_X$ such that for any $k$ and any 
    coherent sheaf of $\Orb_{\AN{Z}}$-ideals $\mathscr{I} \supset 
    (\an{\mathscr{J}})_{\AN{Z}}^k$ the analytic $\an{X}$-stack
    $V(\mathscr{I})$ is algebraizable to a quasi-finite and separated 
    algebraic $X$-stack, and $\ker f^\sharp$ and $\coker
    f^\sharp$ are annihilated by $(\an{\mathscr{I}})_{\AN{Z}}^j$ for
    some $j$, then $\AN{Z}$ is algebraizable to a quasi-finite and
    separated algebraic $X$-stack. 
  \end{enumerate}
\end{lem}
To prove Lemma \ref{lem:inj_surj_analy} and Proposition
\ref{prop:fet_dev}, it will be necessary to understand certain pushouts for
analytic stacks, which we will  
defer until later in this section. We are, however, ready to prove
Proposition \ref{prop:surj_cpoints_dev_case}. 
\begin{proof}[Proof of Proposition \ref{prop:surj_cpoints_dev_case}]
  For a Zariski closed subset $|V| \subset |X|$, let $P_{|V|}$
  be the statement: for any closed subscheme $V_0 \hookrightarrow X$
  such that $|V_0| = |V|$, the functor $\Psi_{V_0}$ is an
  equivalence. Since the statement $P_{\emptyset}$ is trivially true,
  by the principle of noetherian induction and the statement of the
  Proposition, we will have proven the Lemma if we show the following:
  the truth of the statement $P_{|V|}$ for all proper closed subsets
  $|V| \subsetneq |X|$ implies the truth of the statement $P_{|X|}$.

  By hypothesis, there is a proper and birational
  $S$-morphism $p : X' \to X$ such that $\Psi_{X'}$ is an
  equivalence. Consider a dense open subscheme $U \subset X$ for which 
  $p^{-1}U \to U$ is an isomorphism, and let $I$ be a coherent sheaf
  of ideals with support $|X\setminus U|$. Next, suppose we have a  
  locally quasi-finite and separated morphism of analytic spaces $\sigma
  : \AN{Z} \to \an{X}$, where $\AN{Z}$ is proper. The
  assumptions on $X'$ ensures that $\AN{Z}_{\an{X}'}$ is
  algebraizable to a quasi-finite and separated 
  $X'$-scheme. By Lemma \ref{lem:stein_analytic}, we conclude that 
  in the Stein factorization of the morphism $(\an{p})_{\AN{Z}} :
  \AN{Z}_{\an{X}'} \to \AN{Z}$, that $\STEIN{\AN{Z}_{\an{X}'}}$ is
  algebraizable to a quasi-finite and separated
  $X$-scheme. For notational brevity, we set $\AN{Z}' =
  \STEIN{\AN{Z}_{\an{X}'}}$ and let the morphism $f : \AN{Z}' \to 
  \AN{Z}$ be the morphism $\STEINFAC{(\an{p})_{\AN{Z}}}$ from the
  Stein factorization of the morphism $(\an{p})_{\AN{Z}}$. By Lemma
  \ref{lem:analy_noeth_ind}, the kernel and cokernel of the map
  $f^\sharp : \Orb_{\AN{Z}} \to f_*\Orb_{\AN{Z}'}$ are annihilated by
  $(\an{I})^j_{\AN{Z}}$ for some $j$. By noetherian induction,
  $V((\an{I})_{\AN{Z}}^k)$ is algebraizable to a quasi-finite and
  separated $X$-scheme for all $k$. Hence, we may apply
  Lemma \ref{lem:inj_surj_analy}\itemref{lem:inj_surj_analy2} to
  conclude that $\AN{Z}$ is algebraizable to a quasi-finite and
  separated $X$-scheme.

  Hence, given $(\AN{Z},\AN{F}) \in \pscSHV{\an{X}}$, by what we have
  proven, we know that the locally quasi-finite and separated morphism
  $\AN{Z} \to \an{X}$ is algebraizable to a quasi-finite and separated
  morphism of schemes $Z \to X$, where $Z$ is $\C$-proper. By GAGA for
  proper $\C$-schemes \cite[Exp. XII, Thm. 4.4]{SGA1}, we deduce that
  the coherent $\Orb_{\an{Z}}$-module $\AN{F}$ is
  algebraizable. Hence, the functor $\Psi_X$ is essentially surjective.
\end{proof}
Next, observe that given a diagram of ringed spaces $E:=[Z^1
\leftarrow Z^3 \rightarrow Z^2]$,  let the topological space $|Z_4|$
be the colimit 
of the induced diagram $|E|:= [|Z^1| \leftarrow |Z^3| \rightarrow |Z^2|]$ in
the category of topological spaces. We have induced maps $m^i : |Z^i|
\to |Z^4|$, and the colimit of the diagram $E$ in the category of
ringed spaces is the ringed space $Z^4:=(|Z^4|, m^1_*\Orb_{Z^1}
\times_{m^3_*\Orb_{Z^3}}m^2_*\Orb_{Z^2})$. We may promote the 
morphisms of topological spaces  $m^i : |Z^i| \to |Z^4|$ to morphisms
of ringed spaces $m^i : Z^i \to Z^4$. If the ringed 
spaces $Z^i$ are locally ringed, and the maps $m^i$ are morphisms of
locally ringed spaces, then since $|Z^1| \amalg |Z^2| \to |Z^4|$ is
surjective, it is easy to see that $Z^4$ is the colimit of the diagram
$E$ in the category of \emph{locally} ringed spaces. These
observations will be of use when the ringed spaces $Z^i$ are schemes
or analytic spaces. Indeed, to show that a scheme (resp. analytic
space) is a pushout of some schemes (resp. analytic spaces), it will
suffice to show that it is the pushout in the category of ringed
spaces, and the maps involved are all maps of schemes (resp. analytic
spaces), which will typically be clear.  

Now, let $X$ be a locally noetherian algebraic stack, and suppose that we have
quasi-finite and separated morphisms $s^i : Z^i \to X$ for $i=1$, $2$,
$3$. In addition, assume that we have finite $X$-morphisms $t^j : Z^3 
\to Z^j$ 
for $j=1$, $2$. It was shown in \cite[Thm.\
2.10]{JACKRYDHHILBSTKSCH-2010}, that 
the resulting diagram of algebraic $X$-stacks $[Z^1 \xleftarrow{t^1} Z^3
\xrightarrow{t^2} Z^2]$ has a colimit, $Z^4:=Z^1 \amalg_{Z^3} Z^2$, in
the category of quasi-finite, separated, and representable algebraic
$X$-stacks. Let $m^i : Z_i \to Z_4$ denote the resulting $X$-morphisms,
which are finite. It was also shown [\emph{loc.\ cit.}], that the
Zariski topological space $|Z^4|$ was the colimit of the diagram of
topological spaces $[|Z^1| \xleftarrow{t^1} |Z^3| \xrightarrow{t^2}
|Z^2|]$, and that there was an isomorphism of  coherent sheaves
$\Orb_{Z^4} \to m^1_*\Orb_{Z_1} \times_{m^3_*\Orb_{Z_3}}
m^2_*\Orb_{Z_2}$. All of this commutes with flat base change on
$Z_4$. Retaining this notation, we have the main result needed to
prove Lemma \ref{lem:inj_surj_analy}.  
\begin{lem}\label{lem:pushouts_analy}
  If $X$ is a locally of finite type algebraic $\C$-stack with locally
  separated diagonal, then $\an{Z}^4$ is the
  colimit of the diagram $[\an{Z}^1 \xleftarrow{\an{t}^1} \an{Z}^3
  \xrightarrow{\an{t}^2} \an{Z}^2]$ in the category of locally
  quasi-finite and separated analytic
  stacks over $\an{X}$, and remains so after flat base change on
  $X$. 
\end{lem}
\begin{proof}
For $k=1$, $2$, $3$ let $m^k  : Z^k \to Z^4$ denote the canonical
map. First, we assume that $X$ is a scheme. For $l=1$, $\ldots$, $4$,  
  let $\phi_l : \an{Z}^l \to Z^l$ denote the canonical map of ringed
  spaces. We note once and for all 
  that by \cite[Exp. XII, 1.3.1]{SGA1}, the functor
  $\phi_l^*$ from $\Orb_{Z^l}$-modules to $\Orb_{\an{Z}^l}$-modules is
  \emph{exact}. Also, we have a bijection of sets: 
  \[
  |\an{Z}^1|\amalg_{|\an{Z}^3|} |\an{Z}^2| \to
  Z^1(\C)\amalg_{Z^3(\C)} Z^2(\C) \to Z^4(\C) \to |\an{Z}^4|.
  \]
  Thus, we conclude that the canonical, continuous map $\psi : 
  |\an{Z}^1|\amalg_{|\an{Z}^3|}|\an{Z}^2| \to |\an{Z}^4|$ is a
  bijection. Since it is also a proper map, we conclude that $\psi$ is a
  homeomorphism. Also, we have that the natural map $\psi^\sharp :
  \Orb_{\an{Z}^4} \to
  (\an{m})^1_*\Orb_{\an{Z}^1}\times_{(\an{m})^3_*\Orb_{\an{Z}^3}}
  (\an{m})^2_*\Orb_{\an{Z}^2}$  factors as the sequence of bijections:
  \begin{align*}
    \Orb_{\an{Z}^4} &\to
    \phi_4^*(m^1_*\Orb_{Z^1}\times_{m^3_*\Orb_{Z^3}}
    m^2_*\Orb_{Z^2})\\
    &\to \phi_4^*m^1_*\Orb_{Z^1} \times_{\phi_4^*m^3_*\Orb_{Z^3}}
    \phi_4^*m^2_*\Orb_{Z^2} \\
    &\to (\an{m}^1)_*\Orb_{\an{Z}^1}
    \times_{(\an{m}^3)_*\Orb_{\an{Z}^3}} (\an{m}^2)_*\Orb_{\an{Z}^2}.
  \end{align*}
  Hence, we conclude that $\an{Z}^4$ is the colimit of the diagram in
  the category of ringed spaces, and remains so after flat base change
  on $X$. It is clear that this implies that $\an{Z}^4$ is the colimit
  in the category of analytic spaces.

  Next, we assume that $X$ is an algebraic space. Let $X_1 \to X$ be an
  \'etale cover by a scheme, and furthermore take $X_2 = X_1\times_X
  X_1$. Take $Z^k_i = Z^k\times_X X_i$ for $i=1$, $2$ and $k=1$, $2$,
  $3$, $4$. By the case of schemes already considered, we know 
  for $i=1$ and $2$ that the analytic space $\an{(Z^4_i)}$ is the
  colimit of the diagram $[\an{(Z^1_i)} \leftarrow \an{(Z^3_i)}
  \rightarrow \an{(Z^2_i)}]$ in the category of analytic spaces. The
  universal properties furnish us with an \'etale equivalence relation
  of analytic spaces $[\an{(Z^4_2)}\rightrightarrows
  \an{(Z^4_1)}]$, with quotient in the category of analytic spaces
  $\an{Z}^4$. The universal property of the pushouts and \'etale
  descent immediately imply that $\an{Z}^4$ is the colimit in the
  category of analytic spaces. In the case that $X$ is an algebraic
  stack, we may argue exactly the same as in the case where $X$ is an
  algebraic space (but instead use smooth covers), and conclude that
  $\an{Z}^4$ is the colimit in the category of locally quasi-finite
  and separated analytic stacks over $\an{X}$. 
\end{proof}
From here, we may prove a Lemma that is necessary for Proposition
\ref{prop:fet_dev}.
\begin{lem}\label{lem:fet_dev_pushout}
  Fix a finite type Deligne-Mumford $\C$-stack $X$ with quasi-compact
  and separated diagonal. Let $\pi : X^1 \to X$ be a finite and surjective
  morphism and consider an open substack $U \subset X$ such that the
  induced morphism $\pi^{-1}U \to U$ is \'etale.  Fix a coherent
  $\Orb_X$-ideal  $\mathscr{I}$ such that $|X\setminus U| =
  |V(\mathscr{I})|$. Consider a locally quasi-finite and separated
  morphism of analytic stacks $\sigma : \AN{Z} \to \an{X}$, where
  $\AN{Z}$ is proper. Let $X^2 := X^1\times_X X^1$ and for $i=1$ and
  $2$ define $\AN{Z}^i := \AN{Z}\times_{\an{X}} \an{X}^i$. There is an
  induced 
  coequalizer diagram $[\AN{Z}^2 \rightrightarrows \AN{Z}^1]$, with
  the maps appearing finite. Suppose that this diagram has a coequalizer
  $\AN{W}$ in the category of locally quasi-finite and separated analytic
  $\an{X}$-stacks (e.g. if $\AN{Z}^1$ and $\AN{Z}^2$ are
  algebraizable). There is an induced finite and surjective morphism 
  $\eta : \AN{W} \to \AN{Z}$ of analytic stacks. Then there is a $k>0$
  such that the kernel and cokernel of the map $\eta^\sharp$ is annihilated by
  the coherent $\Orb_{\AN{Z}}$-ideal $(\an{\mathscr{I}})^k_{\AN{Z}}$.  
\end{lem}
\begin{proof}
  Consider the open analytic substack $\AN{U}':=\sigma^{-1}(\an{U})$
  of $\AN{Z}$. Let $g_i : \AN{Z}^i \to \AN{Z}$ be the induced maps,
  then the induced map $g_1^{-1}\AN{U}' \to \AN{U}'$ is finite
  \'etale. Hence, we deduce by \'etale descent that the map
  $\eta^{-1}\AN{U}' \to \AN{U}$ is an analytic isomorphism. In
  particular, we deduce that $|\supp \ker \eta^\sharp|$ and $|\supp
  \coker \eta^\sharp|$ are contained in
  $|V((\an{I})_{\AN{Z}})|$. Applying Lemma
  \ref{lem:AR_analy}\itemref{lem:AR_analy1}, we deduce the result. 
\end{proof}
\begin{proof}[Proof of Propostion \ref{prop:fet_dev}]
  As in Proposition \ref{prop:surj_cpoints_dev_case}, we prove the
  result by noetherian induction on the closed substacks of
  $X$. Hence, it suffices to assume that the functor $\Psi_{V}$ is an
  equivalence for any closed substack $V \hookrightarrow X$ such that
  $|V| \subsetneq |X|$.  By assumption, there is a finite and
  generically \'etale map $\pi:X^1 \to X$ such that the functors
  $\Psi_{X^1}$ and $\Psi_{X^1\times_X X^1}$ are equivalences. Fix a
  dense open substack $U \hookrightarrow X$ such that $\pi^{-1}U \to
  U$ is \'etale and let $\mathscr{I}$ be a coherent $\Orb_X$-ideal
  such that $|X\setminus U| = |V(\mathscr{I})|$. Let
  $\AN{Z} \to \an{X}$ be a locally quasi-finite and separated morphism
  with $\AN{Z}$ proper. Let $X^2 = X^1\times_X X^1$ and for $i=1$ and
  $2$ we set $\AN{Z}^i = \AN{Z}\times_{\an{X}} \an{X}^i$. By the 
  hypotheses on $\Psi_{X^1}$ and $\Psi_{X^2}$, the diagram $[\AN{Z}^2
  \rightrightarrows \AN{Z}^1]$ is algebraizable. Hence, by Lemma
  \ref{lem:pushouts_analy}, we conclude that the coequalizer in the
  category of locally quasi-finite and separated $\an{X}$-stacks
  exists, and is algebraizable. We denote this coequalizer by
  $\AN{W}$. By Lemma \ref{lem:fet_dev_pushout}, we deduce that the
  induced map $\eta : \AN{W} \to \AN{Z}$ has the property that $\ker
  \eta^\sharp$ and $\coker \eta^\sharp$ are both annihilated by
  $(\an{\mathscr{I}})^k_{\AN{Z}}$ for some $k>0$. By Lemma 
  \ref{lem:inj_surj_analy}\itemref{lem:inj_surj_analy3}, we deduce
  that the morphism $\AN{Z}\to \an{X}$ is algebraizable to a
  quasi-finite, separated, and representable morphism $Z \to X$, with
  $Z$ $\C$-proper. 

  Thus, given $(\AN{Z},\AN{F}) \in \pscSHV{\an{X}}$, by what we have 
  proven, we know that the locally quasi-finite and separated morphism
  $\AN{Z} \to \an{X}$ is algebraizable. By GAGA for
  proper $\C$-stacks (cf. Theorem \ref{thm:clGAGA}), we deduce that
  the functor $\Psi_X$ is essentially surjective.
\end{proof}
Let $f : \AN{Z}' \to \AN{Z}$ be a finite morphism of analytic
stacks. If $\mathscr{I} \ideal \Orb_{\AN{Z}}$ is a coherent ideal such
that $\mathscr{I} \subset \fml{C}_f$, then the image of $\mathscr{I}$
in $f_*\Orb_{\AN{Z}'}$ generates a $f_*\Orb_{\AN{Z}}$-ideal
$\widetilde{\mathscr{I}} \subset \fml{C}_f$, which lies in the image
of $\Orb_{\AN{Z}}$ (this is a general property of conductors). If
$f^\sharp$ is an injective map, then it is easily verified that
$\widetilde{\mathscr{I}}$ is a coherent $\Orb_{\AN{Z}}$-ideal and
$(\widetilde{\mathscr{I}})_{\AN{Z}'} = \mathscr{I}_{\AN{Z}'}$ as
$\Orb_{\AN{Z}'}$-ideals.  We will
use this notation for the remainder of this section. 
\begin{lem}\label{lem:conductor_anal}
  Let $\AN{X}$ be an analytic stack. Consider a finite morphism of
  locally quasi-finite analytic $\AN{X}$-stacks $f : \AN{Z}' \to 
  \AN{Z}$, such that $f^\sharp : \Orb_{\AN{Z}} \to f_*\Orb_{\AN{Z}'}$
  is injective. Fix an ideal $\mathscr{I}\ideal \Orb_{\AN{Z}}$ such
  that $\mathscr{I} \subset \fml{C}_f$. Then
  the diagram of analytic $\AN{X}$-stacks: 
  \[
  \xymatrix{V(\mathscr{I}_{\AN{Z}'}) \ar@{^(->}[r]
    \ar[d] & \ar[d]^f \AN{Z}' \\ V(\widetilde{\mathscr{I}})
    \ar@{^(->}[r] & \AN{Z}}  
  \]
  is cocartesian in the category of locally quasi-finite analytic
  $\AN{X}$-stacks. 
\end{lem}
\begin{proof}
First, we assume that $\AN{X}$ is an analytic space, and we will show
that the diagram:
\[
 \xymatrix{V(\mathscr{I}_{\AN{Z}'}) \ar@{^(->}[r]
    \ar[d] & \ar[d]^f \AN{Z}' \\ V(\widetilde{\mathscr{I}})
    \ar@{^(->}[r] & \AN{Z}}  
\]
is cocartesian in the category of locally ringed spaces, thus in the
category of analytic spaces. This implies that the diagram is
cocartesian in the category of analytic $\AN{X}$-spaces. We will use
the criterion of \cite[Sc.\ 4.3(b)]{MR2044495}. Note that from the
associated cartesian 
conductor square for \emph{rings}, it suffices to show that
$\AN{Z}$ has the correct topological space. Since $f^\sharp$ is
injective and $f$ is finite, then $f$ is surjective and closed, thus
submersive. Let $\AN{U} = \AN{Z} -
V(\widetilde{\mathscr{I}})$ and $\AN{U}'  =
\AN{Z}'-V(\mathscr{I}_{\AN{Z}'})$. It remains to show that $f$
induces a bijection of sets 
$\AN{U}' \to \AN{U}$. Since $\widetilde{\mathscr{I}} \subset
\fml{C}_f$, then for $u \in \AN{U}$ we have
that the map $f^\sharp_u : \Orb_{\AN{Z},u} \to (f_*\Orb_{\AN{Z'}})_u$ is a
bijection. Thus, since $f$ is finite, we may conclude that the induced
surjective morphism $f^{-1}(\AN{U}) \to
\AN{U}$ has connected fibers---thus is a
bijection of sets.  Hence, we are reduced to showing that the
inclusion $\AN{U}' \hookrightarrow f^{-1}(\AN{U})$ is surjective. This
follows from $(\AN{Z}-V(\widetilde{\mathscr{I}})) \cap
f(V(\mathscr{I}_{\AN{Z}'})) = \emptyset$, which is obvious. 

In the case where $\AN{X}$ is an analytic stack, since all of these
constructions commute with smooth base change on $\AN{Z}$, we 
may work smooth locally on $\AN{X}$ and deduce the result from the
case of analytic spaces already proved. 
\end{proof}
Finally, we arrive at the proofs of Lemma \ref{lem:inj_surj_analy} and
Proposition \ref{prop:etGAGA}.
\begin{proof}[Proof of Lemma \ref{lem:inj_surj_analy}]
  For \itemref{lem:inj_surj_analy2}, we consider the diagram of
  analytic $\an{X}$-stacks $[V(\mathscr{I}) \leftarrow
  V(\mathscr{I}_{\AN{Z}'}) 
  \rightarrow \AN{Z}']$. By hypothesis, these analytic stacks are all
  algebraizable to quasi-finite and separated algebraic
  $X$-stacks and so an application of Lemma \ref{lem:pushouts_analy}
  produces a factorization of the map $f : \AN{Z}' \to \AN{Z}$ into
  $\AN{Z} \to \AN{Z}'' \xrightarrow{\beta} \AN{Z}$, where $\AN{Z}''$
  is algebraizable to a quasi-finite and separated algebraic
  $X$-stack. An easy calculation with rings verifies that $\beta$ has
  the desired properties. 

  For \itemref{lem:inj_surj_analy3}, by Lemma
  \ref{lem:AR_analy}\itemref{lem:AR_analy2} we may replace
  $\mathscr{J}$ by some power $\mathscr{J}^k$ such that
  $(\an{\mathscr{J}})^k_{\AN{Z}} \cap \ker f^\sharp = (0)$.  We now
  observe that \itemref{lem:inj_surj_analy2} applies, and we are
  reduced to the case where $f^\sharp$ is injective. Since there is a
  $l$ such that $\fml{C}_f \supset (\an{\mathscr{J}})^l_{\AN{Z}}$,
  then $V((\an{\widetilde{\mathscr{J}}})^l_{\AN{Z}})$ is algebraizable
  to a quasi-finite and separated algebraic $X$-stack. Moreover,
  $V((\an{\mathscr{J}})^l_{\AN{Z}'})$ is algebraizable to a quasi-finite and
  separated algebraic $X$-stack. Since the following commutative diagram
  \[
  \xymatrix{V((\an{\mathscr{J}})^l_{\AN{Z}'}) \ar[d] \ar[r] &
    \AN{Z}'\ar[d]\\V((\an{\widetilde{\mathscr{J}}})^l_{\AN{Z}})
    \ar[r] & \AN{Z} }
  \]
  is cocartesian in the category of locally quasi-finite analytic
  $\an{X}$-stacks (by Lemma \ref{lem:conductor_anal}), then an
  application of Lemma \ref{lem:pushouts_analy} implies that $\AN{Z}$
  is algebraizable to a quasi-finite and separated algebraic $X$-stack.
\end{proof}
\begin{proof}[Proof of Proposition \ref{prop:etGAGA}] 
  The interesting point here is that we permit the map $W \to Z$ to be
  non-separated. On the 
  small \'etale site of $Z$, we define the sheaf $H_{W/Z} :=
  \Hom_Z(-,W)$. By \cite[Exp. IX, 2.7.1]{SGA4}, since $W \to Z$ is
  quasi-compact and \'etale, the sheaf $H_{W/Z}$ is
  constructible. The analytification of the \'etale sheaf $H_{W/Z}$
  is the sheaf $H_{\an{W}/\an{Z}} := \Hom_{\an{Z}}(-,\an{W})$ on the
  analytic small \'etale site of $\an{Z}$. We now have the comparison
  map on global sections 
  \[
  \Hom_Z(Z,W) = \Gamma(Z,H_{W/Z}) \to \Gamma(\an{Z},
  H_{\an{W}/\an{W}}) = \Hom_{\an{Z}}(\an{Z},\an{W}).
  \]
  which we must show is a bijection. More generally, it suffices to
  show that if $G$ is a constructible sheaf of \emph{sets} on $Z$, then the
  comparison map
  \[
  \delta_G : \Gamma(Z,G) \to \Gamma(\an{Z},\an{G})
  \]
  is bijective. This may be checked \'etale locally on $Z$, so we may
  assume that $Z$ is separated and of finite type over $\C$. The
  constructibility of $G$ guarantees that there is an inclusion $G
  \hookrightarrow G'$, where $G'$ is a constructible sheaf of abelian
  groups on $Z$. By M. Artin's general result on the comparison between
  \'etale and complex cohomology for constructible sheaves of abelian
  groups, the comparison map $\delta_{G'}$ is bijective. One now
  deduces that the comparison map $\delta_G$ is bijective.
\end{proof}
\section{Applications}\label{sec:applications_GAGA}
In this section, we will make Theorem \ref{thm:mod_applications}
precise. Fix an integer $n\geq 0$. Let $\curv_n$ denote the moduli
stack of all $n$-pointed $\C$-curves. That 
is, a morphism from a $\C$-scheme $T$ to $\curv_n$ is equivalent to a
proper, flat, and finitely presented morphism of algebraic $\C$-spaces $C
\to T$ with one-dimensional geometric fibers, together with $n$
sections to the map $C \to T$. It was shown, in
\cite{JACK-2009}, which is an appendix to \cite{smyth-2009}, that the
stack $\curv_n$ is algebraic, locally of finite presentation over $\C$,
with quasi-compact and separated diagonal.
\begin{defn}
  Fix $g>1$, and let ${M}_{g,n}$ denote the moduli stack of smooth
  curves of genus $g$. An \textbf{algebraic modular
    compactification} of ${M}_{g,n}$ is a proper Deligne-Mumford
  $\C$-stack $N$, fitting into a $2$-fiber diagram of algebraic
  $\C$-stacks: 
  \[ 
  \xymatrix{V \ar[r]^{j'} \ar[d]_{i'} & M_{g,n} \ar[d]^i \\ N \ar[r]_j & U_n}, 
  \]
  where the map $j$ is an open immersion, and the maps $i'$ and $j'$
  have dense image. 
\end{defn}
On the analytic side, let $\AN{\curv}_n$ denote the $\ANABS$-stack
of all analytic curves. That is, a map from an analytic space
$\AN{T}$   to $\AN{\curv}_n$ is equivalent to a proper and flat morphism
of analytic spaces $\AN{C} \to \AN{T}$ with one-dimensional
fibers, together with $n$ sections to the map $\AN{C} \to \AN{T}$. 
\begin{defn}
  Fix $g>1$, and let $\AN{M}_{g,n}$ denote the analytic moduli stack
  of smooth, $n$-pointed curves of genus $g$. An \textbf{analytic
    modular compactification} of 
  $\AN{M}_{g,n}$ is a proper Deligne-Mumford analytic stack $\AN{N}$,
  fitting into a $2$-fiber diagram of $\ANABS$-stacks:
 \[ 
  \xymatrix{\AN{V} \ar[r]^{\jmath'} \ar[d]_{\imath'} & \AN{M}_{g,n}
    \ar[d]^{\imath} \\ \AN{N} \ar[r]_{\jmath} & \AN{U}_n},  
  \]
  where the map $\jmath$ is an open immersion, and the maps
  $\imath'$ and $\jmath'$ are dense open immersions. 
\end{defn}
 Clearly, there is a natural morphism of $\ANABS$-stacks:
 $\an{(\curv_n)} \to \AN{U}_n$. We may no claim of originality for the
 following result, but we were unable to find a precise reference.
\begin{thm}\label{thm:anal_curves}
  The morphism of $\ANABS$-stacks $\an{(\curv_n)} \to
  \AN{U}_n$ is an equivalence; hence, $\AN{U}_n$ is an analytic
  stack. Moreover, this equivalence sends $\an{(M_{g,n})}$ to
  $\AN{M}_{g,n}$.   
\end{thm}
\begin{proof}
  The latter claim follows from the first, since smoothness of a
  proper algebraic $\C$-space can be tracked by its
  analytification. We prove the 
  first claim using
  the criteria of Corollary 
  \ref{cor:conrad_crit}. By
  \cite{MR597743}, we know that the diagonal map 
  $\Delta_{\AN{\curv}_0}$ is representable by analytic spaces. Since
  the forgetful map $\AN{U}_n \to \AN{U}_0$ is representable by
  analytic spaces for any $n\geq 0$, we deduce that the diagonal map
  $\Delta_{\AN{\curv}_n}$ is representable by analytic spaces for all
  $n\geq 0$. By Lemma \ref{lem:EXTRAUSEFUL1}, it remains to verify
  that the functor 
  $\curv_n(S) \to \AN{U}_n(\an{S})$ is an equivalence for any local
  artinian $\C$-scheme $S$. Note that the GAGA results for proper
  $\C$-schemes (cf. \cite[Exp. XII, Cor. 4.5]{SGA1}) show that this
  functor is fully 
  faithful. To get essential surjectivity, we note that for a proper and
  flat analytic curve $\AN{C} \to \an{S}$, $\AN{C}$ is a
  one-dimensional proper analytic space, thus is algebraizable. The
  flat structure map $\AN{C} \to \an{S}$ also algebraizes using the
  classical GAGA results [\emph{loc.\ cit.}]\ . The classical GAGA results
  [\emph{loc.\ cit.}]\ also show that
  the sections may be algebraized, which gives the desired equivalence.
 \end{proof}
\begin{proof}[Proof of Theorem \ref{thm:mod_applications}]
  Let $\jmath : \AN{N} \hookrightarrow \AN{\curv}_n$ be an analytic
  modular compactification of $\AN{M}_{g,n}$.  Theorems
  \ref{thm:anal_curves} and \ref{thm:GAGA_stacks_VAR1}, imply that
  the open immersion of analytic stacks $\jmath : \AN{N}
  \hookrightarrow \AN{\curv}_n$ is the algebraizable to an open
  immersion $j : N \hookrightarrow 
  \curv_n$, where  $N$ is a proper Deligne-Mumford $\C$-stack. Next,
  form the 2-fiber square: 
  \[
  \xymatrix{V \ar[r]^{j'} \ar[d]_{i'}& M_{g,n} \ar[d]^i \\ N \ar[r]_j & U_n }.
  \]
  Clearly, all maps in the above diagram are open immersions, and it
  remains to show that $i'$ and $j'$ have dense image. This may be
  checked after passing to the analytifications, and since
  analytification commutes with $2$-fiber products, we're done.
\end{proof}
\appendix
\section{Separated GAGA}\label{app:clGAGA}
For this paper we required a mild strengthening of the classical GAGA
results contained in \cite[Exp. XII]{SGA1}, for which we could not
find a reference for in the existing literature. For a separated and locally
of finite type algebraic $\C$-stack $X$, define $\COHP{X}$ to be the
category of coherent sheaves of $\Orb_X$-modules with proper
support. For a separated analytic stack $\AN{X}$, define
$\COHP{\AN{X}}$ to be the category of coherent sheaves of
$\Orb_{\AN{X}}$-modules with proper support. If $X$ is a separated and
locally of finite type algebraic $\C$-stack, there is an
analytification functor
\[
A_X : \COHP{X} \longrightarrow \COHP{\an{X}}.
\]
If the algebraic stack $X$ is a projective $\C$-scheme, then
the functor $A_X$ was shown to be an equivalence by Serre
\cite{MR0082175}. If the algebraic stack $X$ is assumed to be a proper
$\C$-scheme, then the functor $A_X$ was shown to be an equivalence by
Grothendieck in \cite[Exp. XII]{SGA1}. We were unable to find a
reference proving that the analytification functor $A_X$ is an
equivalence in the case that the algebraic $\C$-stack $X$ is a
quasi-projective $\C$-scheme. For the purposes of this paper, we need
that the functor $A_X$ is an equivalence in the case of a separated
Deligne-Mumford $\C$-stack $X$. Utilizing the work of M. Olsson 
\cite{MR2183251}, however, makes it of equal difficuly to prove the
following    
\begin{thm}[Separated GAGA]\label{thm:clGAGA}
  Let $X$ be a separated and locally of finite type algebraic
  $\C$-stack. Then the analytification functor
  \[
  A_X : \COHP{X} \longrightarrow \COHP{\an{X}}
  \]
  induces an equivalence of categories.
\end{thm}
This readily implies the following corollaries.
\begin{cor}\label{cor:clGAGA1}
  Let $X$ be a separated and locally of finite type algebraic
  $\C$-stack. Given a finite morphism of analytic stacks $\AN{Z} \to
  \an{X}$, where $\AN{Z}$ is proper, then $\AN{Z}$ is 
  uniquely algebraizable to a finite $X$-stack $Z \to X$.
\end{cor}
\begin{cor}\label{cor:clGAGA2}
  Let $X$ be a separated and locally of finite type Deligne-Mumford
  $\C$-stack. Suppose that $Z$ is a proper Deligne-Mumford $\C$-stack,
  then the analytification functor
  \[
  \Hom(Z,X) \longrightarrow \Hom(\an{Z},\an{X})
  \]
  is an equivalence of categories.
\end{cor}
Before we prove Theorem \ref{thm:clGAGA}, we need the following
preliminary result.
\begin{prop}\label{prop:clGAGA_mor}
  Let $f : X \to Y$ be a separated morphism of locally of finite type
  algebraic $\C$-stacks with locally separated diagonals. Fix a
  coherent $\Orb_X$-module $F$, with support proper over $Y$. Then for
  each $i\geq 0$, the analytic comparison map
  \[
  \an{(R^if_*F)} \longrightarrow R^i(\an{f})_*\an{F}
  \]
  is an isomorphism of coherent $\Orb_{\an{Y}}$-modules. In
  particular, if $X$ is a separated and locally of finite type
  algebraic $\C$-stack, then for each $i\geq 0$, the comparison
  map 
  \[
  H^i(X,F) \longrightarrow H^i(\an{X},\an{F})
  \]
  is an isomorphism of $\C$-vector spaces.
\end{prop}
\begin{proof}
  The statement is local on $Y$ for the smooth topology, so we may
  assume henceforth that $Y$ is a separated and finite type
  $\C$-scheme. Next, let $Z$ denote the stack-theoretic support of
  $F$, and let $\imath :Z \to X$ 
  denote the inclusion. Let $g = f\circ \imath$, then since the
  functors $\imath_*$ and $\an{(\imath_*)}$ are exact, we have
  a commutative diagram for every $i\geq 0$:
  \[
  \xymatrix{\an{(R^if_*[\imath_*\imath^*F])} \ar[r] \ar[d] &
    R^i(\an{f})_*(\an{\imath})_*\an{[\imath^*F]}\ar[d] \\
    \an{(R^ig_*[\imath^*F])} \ar[r] & 
    R^i(\an{g})_*\an{[\imath^*F]}. }
  \]
  The two vertical maps are isomorphisms. In particular, since the
  natural maps $\imath_*\imath^*F \to F$ and
  $(\an{\imath})_*\an{[\imath^*F]}$ are isomorphisms, then we are
  reduced to the case where the original morphism $f$ is proper. In the case
  that the morphism $f$ is representable by schemes, the result
  follows from \cite[Exp. XII, Thm. 4.2]{SGA1}. Next, assume that $f$
  is representable by algebraic spaces, then by Chow's Lemma
  \cite[Thm. 4.3.1]{MR0302647}, there is a projective and birational
  $Y$-morphism $g : X' \to X$, such that the composition $X' \to Y$ is
  projective. Repeating verbatim the d\'evissage argument given in the
  proof of \cite[Exp. XII, Thm. 4.2(2)]{SGA1} proves the result in the
  case that the morphism $f$ is representable by algebraic spaces. 
  
  For the general case, we note that standard reductions, combined
  with noetherian induction and d\'evissage allow us to reduce to the
  case where $X$ is reduced, 
  $\supp F = X$ and the result is proven for all coherent sheaves $G$
  on $X$ with $|\supp G| \subsetneq |X|$.  Using
  \cite[Thm. 1.1]{MR2183251}, there is a proper and representable
  $Y$-morphism $h : V \to X$, where $V$ is a projective $Y$-scheme. Since
  $X$ is reduced, then by generic flatness  \cite[\textbf{IV},
  6.9.1]{EGA}, there is a dense open $W \hookrightarrow X$ for which
  the morphism $h^{-1}W \to W$ is flat. Let $h^2 : V\times_X V \to X$
  denote the induced morphism, then there is a natural map
  \[
  \alpha : F \to F' := \mathrm{eq}\,(h_*h^*F \rightrightarrows
  h^2_*(h^2)^*F), 
  \]
  and by fppf descent, it is an isomorphism over $W$. In particular,
  $\ker \alpha$, $\im \alpha$, and $\coker \alpha$ are supported on
  the complement of $|W|$. Since the map $h$ is representable, then by
  what we have already proven, the coherent sheaf $F'$ satisfies the
  conclusion of the Proposition. Using the resulting two exact sequences 
  involving all of these terms, and d\'evissage, we deduce the result.
\end{proof}
\begin{proof}[Proof of Theorem \ref{thm:clGAGA}]
  That the functor $A_X$ is fully faithful is exactly the same as
  \cite[Exp. XII, Thm. 4.4]{SGA1}, where we use Propostion
  \ref{prop:clGAGA_mor} in place of \cite[Exp. XII, Thm. 4.2]{SGA1}. We
  now proceed to show that the functor $A_X$ is essentially
  surjective. An easy reduction allows us to immediately reduce to the
  case where the algebraic stack $X$ is assumed to be quasi-compact.\\  
  \emph{Quasi-projective case:} Consider a compactification $\jmath :
  X \hookrightarrow \bar{X}$, where $\bar{X}$ is a projective
  $\C$-scheme. Let $\AN{F} \in \COHP{X}$, then $\AN{Z}:=V(\Ann\AN{F})$
  is a proper analytic space, which is a closed analytic subspace of
  $\an{\bar{X}}$. By \cite[Exp. XII, Thm. 4.4]{SGA1}, we conclude that
  $\AN{Z}$ is algebraizable to a projective $\C$-scheme $Z
  \hookrightarrow \bar{X}$.  By passing to analytifications, it is
  easy to verify that the map $Z \to \bar{X}$ factors through
  $X$. Let $i : Z \hookrightarrow X$ denote the resulting closed
  immerion of $\C$-schemes, then since analytification preserves
  annihilators we have that the natural map $(\an{i})_*\an{i}^*\AN{F}
  \to \AN{F}$ is an isomorphism; and it suffices to show that
  $\an{i}^*\AN{F} \in \COH{\an{Z}}$ lies in the essential image of
  $A_Z$. This is clear from the usual GAGA statements \cite[Exp. XII,
  Thm. 4.4]{SGA1}.\\
  \emph{General case:} This is a standard d\'evissage argument,
  similar to that used in the proof of Proposition
  \ref{prop:clGAGA_mor}, so we omit it.  
\end{proof}

\newcommand{\etalchar}[1]{$^{#1}$}
\providecommand{\bysame}{\leavevmode\hbox to3em{\hrulefill}\thinspace}
\providecommand{\MR}{\relax\ifhmode\unskip\space\fi MR }
\providecommand{\MRhref}[2]{%
  \href{http://www.ams.org/mathscinet-getitem?mr=#1}{#2}
}
\providecommand{\href}[2]{#2}

\end{document}